\documentclass{article}
\usepackage{amsthm,amsmath, amssymb, amsfonts}
\newtheorem{theorem}{Theorem}
\newtheorem{lemma}[theorem]{Lemma}

\newtheorem{proposition}[theorem]{Proposition}

\theoremstyle{Example}
\newtheorem{question}{Question}

\newcommand*{\R}{\mathbb{R}}
\newcommand*{\Z}{\mathbb{Z}}

\newcommand*{\Zp}{\mathbb{Z}_{\geq 0}}
\newcommand*{\Zb}{\bar{\mathbb{Z}}}

\newcommand*{\pset}{\mathcal{P}}

\newcommand*{\aff}{\mathrm{aff.hull}}
\newcommand*{\conv}{\mathrm{conv.hull}}
\newcommand*{\crank}{\mathrm{cr}}
\newcommand*{\Car}{Carath\'eodory}
\newcommand*{\transp}{\mathsf{T}}

\title{Polyhedra with the Integer \Car \! \! Property}
\author{Dion Gijswijt\footnote{CWI and Dep. of Mathematics, Leiden University. Email: dion.gijswijt@gmail.com.}\ and Guus Regts\footnote{CWI, Amsterdam. Email: regts@cwi.nl.}}

\begin{document}
\maketitle
\begin{abstract}
A polyhedron $P$ has the \emph{Integer Carath\'eodory Property} if the following holds. 
For any positive integer $k$ and any integer vector $w\in kP$, there exist affinely independent integer vectors $x_1,\ldots,x_t\in P$ and positive integers $n_1,\ldots,n_t$ such that $n_1+\cdots+n_t=k$ and $w=n_1x_1+\cdots+n_tx_t$. 
In this paper we prove that if $P$ is a (poly)matroid base polytope or if $P$ is defined by a TU matrix, then $P$ and projections of $P$ satisfy the integer \Car \! \! property.
\\[.5cm]
\end{abstract}

\section{Introduction}
A polyhedron $P\subseteq \R^n$ has the \emph{integer decomposition property} if for every positive integer $k$, every integer vector in $kP$ is the sum of $k$ integer vectors in $P$. Equivalently, every $\tfrac{1}{k}$-integer vector $x\in P$ is a convex combination
\begin{equation}\label{fractional decomposition}
x=\lambda_1x_1+\cdots+\lambda_tx_t,\quad x_i\in P\cap \Z^n, \lambda_i\in \tfrac{1}{k}\Z.
\end{equation}
Examples of such polyhedra include: stable set polytopes of perfect graphs, polyhedra defined by totally unimodular matrices and matroid base polytopes. 

It is worth remarking the relation with Hilbert bases. Recall that a finite set of integer vectors $H$ is called a \emph{Hilbert base} if every integer vector in the convex cone generated by $H$, is an integer sum of elements from $H$. Hence if $P$ is an integer polytope and $H:=\{\left(\begin{smallmatrix}1\\x\end{smallmatrix}\right)\mid x\in P\text { integer}\}$, then $P$ has the integer decomposition property, if and only if $H$ is a Hilbert base.

Let $P$ be a polytope with the integer decomposition property. It is natural to ask for the smallest number $T$, such that we can take $t\leq T$ in (\ref{fractional decomposition}) for every $k$ and every $\tfrac{1}{k}$-integer vector $x\in P$. We denote this number by $\crank(P)$, the \emph{Carath\'eodory rank} of $P$.  
Clearly, $\crank(P)\geq \dim (P)+1$ holds, since $P$ is not contained in the union of the finitely many affine spaces spanned by at most $\dim (P)$ integer points in $P$. 


Cook et al. \cite{CookFonluptSchrijver} showed that when $H$ is a Hilbert base generating a pointed cone $C$ of dimension $n$, every integer vector in $C$ is the integer linear combination of at most $2n-1$ different elements from $H$. For $n>0$, this bound was improved to $2n-2$ by Seb\H o \cite{Sebo}. By the above remark, this implies that $\crank(P)\leq 2\dim (P)$ holds for any polytope $P$ of positive dimension.

Bruns et al. \cite{Brunsetal} give an example of a Hilbert base $H$ generating a pointed cone $C$ of dimension $6$, together with an integer vector in $C$ that cannot be written as a nonnegative integer combination of less than $7$ elements from $H$. Their example yields a $0-1$ polytope with the integer decomposition property of dimension $5$ but with Carath\'eodory rank $7$, showing that $\crank(P)=\dim(P)+1$ does not always hold. The vertices of the polytope are given by the columns of the matrix
\begin{equation}
\left[\begin{array}{rrrrrrrrrr}
1&1&1&1&1&0&0&0&0&0\\
1&1&0&0&0&0&0&1&0&1\\
0&1&1&0&0&1&0&0&1&0\\
0&0&1&1&0&0&1&0&0&1\\
0&0&0&1&1&1&0&1&0&0
\end{array}\right].
\end{equation}

In this paper we prove that if $P$ is a (poly)matroid base polytope or if $P$ is a polyhedron defined by a TU matrix then $P$ and projections of $P$ satisfy the inequality $\crank(P)\leq \dim(P)+1$. For matroid base polytopes this answers a question of Cunningham \cite{Cunningham} asking whether a sum of bases in a matroid can always be written as a sum using at most $n$ bases, where $n$ is the cardinality of the ground set (see also \cite{Sebo,EGRES}). 

In our proof we use the following strengthening of the integer decomposition property, inspired by Carath\'eodory's theorem from convex geometry. We say that a polyhedron $P\subset \R^n$ has the \emph{Integer Carath\'eodory Property} (notation: ICP) if for every positive integer $k$ and every integer vector $w\in kP$ there exist affinely independent $x_1,\ldots, x_t\in P\cap \Z^n$ and $n_1,\ldots,n_t\in \Zp$ such that $n_1+\cdots+n_t=k$ and $w=\sum_i n_ix_i$. Equivalently, the vectors $x_i$ in (\ref{fractional decomposition}) can be taken to be affinely independent. In particular, if $P$ has the ICP, then $\crank(P)\leq \dim P+1$.

It is implicit in \cite{CookFonluptSchrijver, Sebo} that the stable set polytope of a perfect graph has the ICP since a `Greedy' decomposition can be found, where the $x_i$ are in the interior of faces of decreasing dimension, and hence are affinely independent. 

The organization of the paper is as follows. In Section 2 we introduce an abstract class of polyhedra and show that they have the ICP. 

In Section 3 we apply this result to show that polyhedra defined by (nearly) totally unimodular matrices, and their projections, have the ICP.

Section 4 deals with applications to (poly)matroid base polytopes and the intersections of two gammoid base polytopes, showing that these all have the ICP. We conclude by stating two open problems related to matroid intersection.

\section{A class of polyhedra having the ICP}
In this section we give a sufficient condition for a polyhedron $P\subset \R^n$ to have the Integer \Car \! \! Property.
This condition is closely related to the \emph{middle integral decomposition condition} introduced by McDiarmid in \cite{mcdiarmid}.
First we introduce some notation and definitions.

Throughout this paper we set $\Zb:=\Z\cup\{-\infty,+\infty\}.$
For a vector $x\in \R^n$ we denote the $i$-th entry of $x$ by $x(i)$.
Recall that a polyhedron $P$ is called an \emph{integer polyhedron} if every face of $P$ contains an integer point.
The polyhedron $P$ is called \emph{box-integer} if for every pair of vectors $c\leq d\in \Zb^n$, the set $\{x\in P|c\leq x\leq d\}$ is an integer polyhedron.

Let $\pset$ be the set of rational polyhedra $P\subseteq \R^n$ (for some $n$) satisfying the following condition:
\begin{equation}
\begin{array}{c}
\text{For any }k\in \Zp, r\in \{0,\ldots,k\}  \text{ and } w\in \Z^n   
\\
\text{ the intersection } rP\cap (w-(k-r)P) \text{ is box-integer.}  \label{eq:condition}
\end{array}
\end{equation}

\begin{theorem} \label{basic theorem}
If $P\in \pset$, then P has the Integer \Car \! \! Property.
\end{theorem}

Before we prove this theorem we first need a few results describing some properties of $\pset$.

\begin{lemma}\label{int decomp}
Every $P\in \pset$ has the integer decomposition property.
\end{lemma}

\begin{proof}
Let $P\subseteq \R^n$ be in $\pset,$ let $k$ be a positive integer and let $w\in kP\cap \Z^n$.
Note that $P\cap (w-(k-1)P)$ is not empty since it contains $\tfrac{1}{k}w=w-\tfrac{(k-1)w}{k}$.

Since $P\in \pset$, the intersection $P\cap (w-(k-1)P)$ is box integer.
Take any integer point $x_k\in P\cap (w-(k-1)P)$ and note that $w-x_k\in (k-1)P\cap \Z^n$.
So by induction we can write $w=(x_1+\ldots+x_{k-1})+x_k$ with $x_i \in P\cap \Z^n$ for all $i$. 
\end{proof}

As a consequence of Lemma \ref{int decomp}, every $P\in \pset$ is an integer polyhedron.
Indeed, let $F$ be a face of $P$ and $x\in F$ a rational point. 
Take $k\in \Zp$ such that $kx\in \Z^n$. By Lemma \ref{int decomp} we can write $kx=\sum_{i=1}^k x_i$ with $x_i\in P\cap \Z^n$. Clearly, $x_1,\ldots, x_k\in F$, hence $F$ contains an integer vector.

\begin{lemma}\label{properties of P}
The collection $\pset$ is closed under taking faces and intersections with a box. 
\end{lemma}

\begin{proof}
First note that if $P_1$ and $P_2$ are two polyhedra and $F_i\subseteq P_i$ are faces, then either $F_1\cap F_2=\emptyset$ or $F_1\cap F_2$ is a face of $P_1\cap P_2$.
 
Now let $P\in \pset$ and let $F$ be a face of $P$. 
To see that $F$ satisfies \eqref{eq:condition}, let $k\in \Zp, r\in \{0,\ldots,k\}$ and let $w\in\Z^n$. 
If $rF\cap (w-(k-r)F)$ is empty there is nothing to prove. Otherwise, it is a face of $rP\cap (w-(k-r)P)$ and hence box-integer.

To see the second assertion, let $c\leq d\in \Zb^n$ and consider the polyhedron $P':=\{x\in P\mid c\leq x\leq d\}$. Let $w\in \Z^n$, $k\in \Zp$ and $r\in \{0,\ldots,k\}$.
Note that 
$rP'\cap (w-(k-r)P')$ is equal to
\begin{equation}
\{x\in rP\cap (w-(k-r)P)\mid rc\leq x\leq rd, w-(k-r)d\leq x\leq w-(k-r)c\}.
\end{equation}
Hence $rP'\cap (w-(k-r)P')$ is the intersection of the box-integer polyhedron $rP\cap (w-(k-r)P)$ with a box, which is again box-integer.
\end{proof}

Note that Lemma \ref{properties of P} together with the observation below Lemma \ref{int decomp} imply the following.

\begin{proposition} 	\label{box int}
Every $P\in \pset$ is a box-integer.
\end{proposition}

We are now ready to prove Theorem \ref{basic theorem}.
\begin{proof}[Proof of Theorem \ref{basic theorem}.]
Let $P\subseteq \R^n$ be a polyhedron in $\pset$.
The proof is by induction on $\dim(P)$. 

The case $\dim(P)=0$ is clear, so we may assume $\dim(P)\geq 1$.
Let $k$ be a positive integer and let $w\in kP\cap \Z^n$. We may assume that $P$ is polytope by replacing $P$ by 
\begin{equation}
\{x\in P\mid \lfloor\tfrac{1}{k}w(i)\rfloor\leq x(i)\leq \lceil\tfrac{1}{k}w(i)\rceil \text{ for all $i$}\}.
\end{equation}
If $w(i)$ is a multiple of $k$ for each $i=1,\ldots, n$, we write $w=k\cdot \tfrac{1}{k}w$ and we are done.
We may therefore assume that $k$ does not divide $w(n)$ and write $w(n)=kq+r$ with $q\in \Z$ and $r\in \{1,\ldots k-1\}$.

Note that $w\in k(\{x\in P\mid q\leq x(n)\leq q+1\})$.
So by Lemma \ref{int decomp} and \ref{properties of P} we can write $w=x_1+\ldots+ x_k$ with $x_i\in P\cap\Z^n$ and with $q\leq x_i(n)\leq q+1$ for all $i$.
We may assume that $x_i(n)=q+1$ for $i=1,\ldots, r$ and $x_i(n)=q$ for $i=r+1,\ldots,k$.

We denote $P_1:=\{x\in P\mid x(n)=q+1\}$ and $P_2:=\{x\in P\mid x(n)=q\}$.
Set $w':=x_1+\ldots +x_r$. This gives a decomposition of $w$ into two integer vectors
\begin{eqnarray}
w'&\in & rP_1,	\nonumber
\\
w-w'= x_{r+1}+\ldots+x_{k}&\in& (k-r)P_2.
\end{eqnarray}

Define 
\begin{eqnarray}
Q&:=&rP_1\cap (w-(k-r)P_2)\nonumber\\
&=&\left(rP\cap (w-(k-r)P)\right)\cap\{x\mid x(n)=q+1\}
\end{eqnarray}
and note that $Q$ is non-empty as it contains $w'$. 
Let $y\in Q$ be an integral vertex.
Such a $y$ exists because $rP\cap (w-(k-r)P)$ is box-integer by assumption.
Let $F_1$ be the inclusionwise minimal face of $rP_1$ containing $y$ and let $F_2$ be the inclusionwise minimal face of $w-(k-r)P_2$ containing $y$.
Let $H_i=\aff(F_i)$.
Then
\begin{equation}\label{eq:H1 intersection H2}
H_1\cap H_2=\{y\}.
\end{equation}
Indeed, every supporting hyperplane of $rP_1$ containing $y$ should also contain $F_1$, by minimality of $F_1$ hence it contains $H_1$.
Similarly, every supporting hyperplane of $w-(k-r)P_2$ containing $y$ also contains $H_2$.
Since $y$ is a vertex of $Q$, it is the intersection of the supporting hyperplanes of the two polytopes containing $y$ and the claim follows.

Let $F_i'$ be the face of $P_i$ corresponding to $F_i$ ($i=1,2$). That is: $F_1=rF_1'$ and $F_2=w-(k-r)F_2'$.
Since $\dim F_i'\leq \dim P_i\leq \dim P-1$, we inductively obtain integer decompositions 
\begin{equation}
y=m_1x_1+\cdots+m_sx_s,\qquad w-y=n_1y_1+\cdots+n_ty_t,
\end{equation}
where $x_1,\ldots,x_s\in F_1'$ are affinely independent integer vectors, $y_1,\ldots, y_t\in F_2'$ are affinely independent integer vectors and $m_1+\cdots+m_s=r,\ n_1+\cdots+n_t=k-r$. 

To complete the proof, we show that $x_1,\ldots,x_s,y_1,\ldots,y_t$ are affinely independent. Suppose there is an affine dependence 
\begin{equation}
\sum_{i=1}^s\lambda_ix_i+\sum_{i=1}^{t}\mu_iy_i=0,\qquad \sum_i\lambda_i+\sum_i\mu_i=0.
\end{equation}
We need to show that all $\lambda_i$ and all $\mu_i$ are zero.

By considering the last coordinate, we see that $(q+1)\sum_i\lambda_i+q\sum_i\mu_i=0$ and hence $\sum_i\lambda_i=\sum_i\mu_i=0$.

Since $y,rx_1,\ldots,rx_s\in F_1$ and $\sum_i\tfrac{\lambda_i}{r}=0$, it follows from 
\begin{equation}
y+\sum_i\lambda_ix_i=y+\sum_i\tfrac{\lambda_i}{r}(rx_i)
\end{equation}
that $y+\sum_i\lambda_ix_i$ is in the affine hull $H_1$ of $F_1$.
Similarly, $y+\sum_i\lambda_ix_i$ is in the affine hull $H_2$ of $F_2$, since $y,w-(k-r)y_1,\ldots,w-(k-r)y_t\in F_2$ and
\begin{equation}
y+\sum_i\lambda_ix_i=y-\sum_i\mu_iy_i=y+\sum_i\tfrac{\mu_i}{k-r}(w-(k-r)y_i).
\end{equation}
It follows by \eqref{eq:H1 intersection H2} that $y+\sum_i\lambda_ix_i=y$. By affine independence of the $x_i$, this implies that $\lambda_1=\cdots=\lambda_s=0$. Hence $\sum_i\mu_iy_i=0$, which implies by affine independence of the $y_i$ that $\mu_1=\cdots=\mu_t=0$.
\end{proof}

We end this section by showing that projections of polyhedra in $\pset$ also have the ICP.

\begin{theorem}\label{projection}
Let $m\leq n$ and let $\pi:\R^n \to \R^m$ be the projection onto the first $m$ coordinates.
If $P\subset \R^n$ and $P\in \pset$, then $\pi(P)$ has the ICP.
\end{theorem}

\begin{proof}
Define $Q:=\pi(P)$.
Let $k$ be a positive integer and let $w\in kQ\cap \Z^m$.
We may assume that $P$ is bounded.
Indeed, taking $N\in \Zp$ large enough such that $\pi^{-1}(\{w\})\cap [-kN, kN]^n$ is not empty.
We can replace $P$ by 
\begin{equation}
P\cap [-N,N]^n,
\end{equation}
and replace $Q$ by $\pi(P\cap [-N,N]^n)$.

Let $F\subseteq kP$ be an inclusionwise minimal face intersecting $\pi^{-1}(\{w\})$.
Then 
\begin{equation}
\pi|_F \text{ is injective. }		\label{eq:injective}
\end{equation}
Indeed, suppose that $\pi(a)=\pi(b)$ for distinct $a,b\in F$. Let $x\in F\cap \pi^{-1}(\{w\})$. Then since $F$ is bounded, the line $x+\R(b-a)$ intersects $F$ in a smaller face, contradicting the minimality of $F$.

Now note that $F\cap \pi^{-1}(\{w\})$ is the intersection of $F$ with the box 
\begin{equation}
\{x\in \R^n\mid x(i)=w(i), i=1,\ldots,m\}.
\end{equation}
Since $P\in \pset$, also $kP\in \pset$ and so $kP$ is box-integer by Proposition \ref{box int}.
This in turn implies that $F$ is box-integer. 
Hence we can lift $w$ to an integer vector $\hat w\in F\cap \pi^{-1}(\{w\})$.
By Theorem \ref{basic theorem} we can find affinely independent integer vectors $x_1,\ldots,x_t$ in $\tfrac{1}{k}F$ and positive integers $n_1,\ldots,n_t$ such that $n_1+\cdots+n_t=k$ and 
\begin{equation}
\hat w=\sum_{i=1}^t n_ix_i .
\end{equation}
Since $\pi|_F$ is injective, $\pi(x_1),\ldots,\pi(x_t)$ are also affinely independent. Hence 
\begin{equation}
w=\sum_{i=1}^t n_i \pi(x_i)
\end{equation}
is the desired decomposition of $w$. 
\end{proof}

\section{Polyhedra defined by totally unimodular matrices}
In this section we prove that polyhedra defined by (nearly) totally unimodular matrices have the ICP.
Recall that a matrix $A$ is called \emph{totally unimodular} (notation: TU) if for each square submatrix $C$ of $A$, $\det(C)\in\{-1,0,1\}.$ 
For details on TU matrices we refer to \cite{schrijver lin int}.

\begin{theorem}	 \label{tu car}
Let $P:=\{x\in \R^n \mid Ax\leq b\}$, where $A$ is an $m\times n$ TU matrix and $b\in \Z^m$.
Then $P\in \pset$. In particular, every projection of $P$ has the ICP.
\end{theorem}
\begin{proof}
Since the matrix $\left[A^\transp -A^\transp\ \ I\ -I\right]^\transp$ is TU, it follows that $rP\cap (w-(k-r)P)$ is box-integer for any $w\in \Z^m$ and positive integers $r<k$.
Hence $P\in \pset.$

Theorem \ref{projection} now implies that every projection of $P$ has the ICP.
\end{proof}

A consequence of Theorem \ref{tu car} is that co-flow polyhedra, introduced by Cameron and Edmonds in \cite{cameron edmonds}, have the ICP since they are projections of TU polyhedra, as was shown by Seb\H o in  \cite{sebo2}.

We end this section with an extension of Theorem \ref{tu car} to so-called nearly totally unimodular matrices.
In \cite{Dion} a matrix $A$ is called \emph{nearly totally unimodular} (notation: NTU) if there exists a TU matrix $\hat A$ a row $a$ of $\hat A$ and an integer vector $c$ such that $A=\hat A+ca^\transp. $
For a $m\times n$ NTU matrix $A$ and an integer vector $b$ the integer polyhedron $P_{A,b}$ is defined by
\begin{equation}
P_{A,b}:=\conv(\{z\in \Z^n\mid Ax\leq b\}).
\end{equation}
Note that 
\begin{equation}
P_{A,b}=\conv \big( \bigcup_{s\in \Z} \{y\mid  \hat Ay\leq b-sc, a^\transp y=s\} \big).      \label{eq:PAb}
\end{equation}
In order to show $P_{A,b}$ has the ICP, we will use the following theorem from \cite{Dion}.
 
\begin{theorem} \label{dion ntu}
Let $\hat A$ be a $m\times n$ TU matrix let $a$ be a row of $\hat A$, let $b,c\in \Z^m$ and define $A:=\hat A+ca^\transp$.
Let $k$ be a nonnegative integer and let $w\in \Z^n$.
Write $a^\transp w=qk+r$ with $q,r\in \Z$ and with $0\leq r\leq k-1$.
Equivalent are:
\begin{description}
\item[(i)] $w\in kP_{A,b}$
\item[(ii)]the system
\begin{equation}	\label{eq:system}
\begin{array}{ccc}
\hat Ay&\leq&r(b-(q+1)c)
\\
\hat Ay& \geq & \hat A w+(k-r)(qc-b)
\\
a^\transp y &= & r(q+1)
\end{array}
\end{equation}
is feasible.
\end{description}
\end{theorem}

\begin{theorem}		\label{ntu car}
Let $\hat A$ be a $m\times n$ TU matrix let $a$ be a row of $\hat A$, let $b,c\in \Z^m$ and define $A:=\hat A+ca^\transp$.
Then $P_{A,b}$ has the ICP.
\end{theorem}

\begin{proof}
Let $k$ be a positive integer and let $w\in kP_{A,b}$. Write $a^\transp w=qk+r$ with $q,r\in \Z$ and with $0\leq r\leq k-1$. We may assume that $P_{A,b}$ is bounded. Indeed, by Theorem \ref{dion ntu} we can take a solution $y$ of (\ref{eq:system}) and let $l,u\in \Z^n$ be such that $rl\leq y\leq ru$ and $w-(k-r)u\leq y\leq w-(k-r)l$. Define the NTU matrix $A'$ and integer vector $b'$ by
\begin{equation}
A':=\left[\begin{array}{r}A\\I\\-I\end{array}\right]\qquad b':=\left[\begin{array}{r}b\\u\\-l\end{array}\right].
\end{equation}
By Theorem \ref{dion ntu} it follows that $w\in kP_{A',b'}$. Since $P_{A',b'}\subseteq P_{A,b}$ is bounded, we can replace $P_{A,b}$ by $P_{A',b'}$. 

Define polyhedra $P_i\subseteq P_{A,b}$ by
 \begin{align}
 P_1:=&\{y\in \R^n\mid \hat Ay\leq b-(q+1)c, a^\transp y=q+1\}	\nonumber
 \\
 P_2:=&\{y\in \R^n\mid \hat Ay\leq b-qc, a^\transp y=q\}.
 \end{align}

If $r=0$ then $w\in kP_2$ and then the claim follows directly from Theorem \ref{tu car}.
So we may assume $r>0$.

Note that the polyhedron defined by \eqref{eq:system} is equal to $rP_1\cap (w-(k-r)P_2)$ and is nonempty by Theorem \ref{dion ntu}.
Let $y$ be a vertex of $rP_1\cap (w-(k-r)P_2)$. Then $y$ is an integral vector because \eqref{eq:system} is defined by a TU matrix.
Let $F_1\subseteq rP_1$ and $F_2\subseteq w-(k-r)P_2$ be the inclusionwise minimal faces containing $y$.

So we now have a decomposition of $w=y+(w-y)$ with $y\in F_1$ and $w-y\in w-F_2$.
Since $P_1$ and $P_2$ are polytopes defined by TU matrices, Theorem \ref{tu car} implies that we can find a positive integer decomposition of $y$ into affinely independent integer vectors $x_1,\ldots,x_t$ from $\tfrac{1}{r}F_1$ and of $w-y$ into affinely independent integer vectors $y_1,\ldots, y_s$ from $\tfrac{1}{k-r}(w-F_2)$.
 
Completely similar to the proof of Theorem \ref{basic theorem} it follows that $x_1,\ldots,x_t$,\\
$y_1,\ldots,y_s$ are affinely independent. Hence combining the decompositions for $y$ and $w-y$ gives the desired decomposition for $w$.
\end{proof}

Interestingly enough, not every polytope defined by a NTU matrix is contained in $\pset$.
Consider the following example.
Let $P:=\{x\in \R^2|x\geq 0, x_1+2x_2 \leq 2\}$. 
This is an integer polytope, but not box integer (take intersection with $x_1\leq 1$).
But $P$ is defined by a NTU matrix.
Namely, define 
\begin{equation}
A:=\left[\begin{array}{rr}-1&0\\0&-1\\1&2\end{array}\right],\quad b:=\left[\begin{array}{r}0\\0\\2\end{array}\right],
\end{equation}
then $P=P_{A,b}$.

This shows that there exist polytopes having the ICP, which are not projections of polytopes in $\pset$, as box-integrality is maintained under projections.

\section{The (poly)matroid base polytope}
In his paper on testing membership in matroid polyhedra, Cunningham \cite{Cunningham} asked for an upper bound on the number of different bases needed in a representation of a vector as a nonnegative integer sum of bases. It follows from Edmonds matroid partitioning theorem \cite{Edmonds} that the incidence vectors of matroid bases form a Hilbert base for the pointed cone they generate. Hence denoting by $n$ the size of the ground set of the matroid, the upper bound of $2n-2$ applies by Seb\H o \cite{Sebo}. This bound was improved by de Pina and Soares \cite{dePina} to $n+r-1$, where $r$ is the rank of the matroid. Chaourar \cite{Chaourar} showed that an upper bound of $n$ holds for a certain minor closed class of matroids.

In this section we show that the (poly)matroid base polytope has the ICP. This in particularly implies that the upper bound of $n$ holds for all matroids. 
Furthermore, we show that the intersection of any two gammoid base polytopes has the ICP.

First we introduce the basic notions concerning submodular functions. For background and more details, we refer the reader to \cite{Fujishige,Schrijver}.

Let $E$ be a finite set and denote its power set by $2^E$. A function $f:2^E\to \Z$ is called \emph{submodular} if for any $A,B\subseteq E$ the inequality $f(A)+f(B)\geq f(A\cup B)+f(A\cap B)$ holds. 
A function $g:2^E \to \Z$ is called \emph{supermodular} if $-g$ is submodular.
Consider the following polyhedra
\begin{align}
EP_f:=&\{x\in \R^E\mid x(U)\leq f(U)\text{ for all }U\subseteq E\}	\nonumber
\\
P_f:=&\{x\in EP_f \mid  x(U)\geq 0 \text{ for all } U\subseteq E\}
\\
B_f:=&\{x\in EP_f\mid x(E)=f(E)\}.	\nonumber
\end{align}
The polyhedron $EP_F$ is called the \emph{extended polymatroid} associated to $f,$ $P_f$ is called the \emph{polymatroid} associated to $f$ and $B_f$ is called the \emph{base polytope} of $f$. 
Observe that $B_f$ is indeed a polytope, since for $x\in B_f$ and $e\in E$, the inequalities $f(E)-f(E-e)\leq x(e)\leq f(\{e\})$ hold, showing that $B_f$ is bounded. 

A submodular function $f:\pset(E)\to \Z$ is the rank function of a matroid $M$ on $E$ if and only if $f$ is nonnegative, nondecreasing and $f(U)\leq |U|$ for every set $U\subseteq E$. 
In that case, $B_f$ is the convex hull of the incidence vectors of the bases of $M$.

Our main tool for proving that $EP_f$ has the ICP is the following result from \cite{Schrijver}, which is similar to Edmonds' polymatroid intersection theorem \cite{Edmonds}.

\begin{theorem}\label{sub super}
Let $f,g:2^E\to \Z$ be two set functions.
If $f$ is submodular and $g$ is supermodular,
then
\begin{equation}
\{x\in \R^E\mid g(U) \leq x(U)\leq f(U), \text{ for all } U\subseteq E\}	\label{eq:sub super int}
\end{equation}
is box-integer.
\end{theorem}

Theorem \ref{sub super} implies that the extended polymatroid is contained in $\pset$ and hence has the ICP.

\begin{theorem}\label{car polymatroid}		
Let $E$ be a finite set and let $f:2^E \to \Z$ be a submodular function, then $EP_f, P_f, B_f\in \pset$. In particular, each of these polyhedra and their projections have the ICP.
\end{theorem}

\begin{proof}
By Theorem \ref{projection}, it suffices to prove the first part of the theorem. Since $B_f$ is a face of $EP_f$ and $P_f$ is the intersection of $EP_f$ with a box, it suffices by Lemma \ref{properties of P} to prove that $EP_f\in \pset$.

Let $k\in \Zp$, $r\in\{0,\ldots,r\}$ and $w\in \Z^E$. 
First note that $rEP_f=EP_{rf},$ with $rf$ submodular again.
Secondly, let $g:=-(k-r)f+w$ and note that $g$ is supermodular. Observe that for $x\in\R^E$ we have $x\in w-(k-r)EP_f$ if and only if $x(U)\leq g(U)$ for all $U\subseteq E$. Hence $rEP_f\cap (w-(k-r)EP_f)$ is box-integer by Theorem \ref{sub super}.
So indeed, $EP_f\in \pset.$
the ICP.
\end{proof}

Note that Theorem \ref{car polymatroid} implies that generalized polymatroid base polytopes also have the ICP, as they are projections of base polytopes of polymatroids. See \cite{Fujishige} for more details on generalized polymatroids.

Below we show that if $P$ is the intersection of two base polytopes of gammoids, then $P$ has the ICP.

Given a digraph $D=(V,A)$ and subsets $U,S$ of $V$, one can define a matroid on the set $S$ as follows. A subset $I\subseteq S$ is independent if there exists $I'\subseteq U$ with $|I|=|I'|$ and if there are $|I|$ vertex-disjoint (directed) paths from $I'$ to $I$. 
A matroid isomorphic to a matroid defined in this way is called a \emph{gammoid}.
Equivalently, gammoids are restrictions of duals of transversal matroids. 
See \cite{Schrijver} for more details on gammoids.

We have the following theorem.
\begin{theorem}		\label{gammoids icp}
Let $P_1$ and $P_2$ be the base polytopes of two gammoids $M_1$ and $M_2$ of rank $k$ defined on the same ground set $S$. Let $P:=P_1\cap P_2$, then $P$ has the ICP.
\end{theorem}
\begin{proof} 
For $i=1,2$, let $M_i$ be associated to digraph $D_i=(V_i,A_i)$ induced by sets $U_i,S_i$. We may assume that $V_1$ and $V_2$ are disjoint. We may further assume that $S=S_1$ and denote by $\varphi:S_1 \to S_2$ the bijection corresponding to the identifiction of $S_2$ and $S$.

We define a new digraph by glueing $D_1$ to the reverse of $D_2$ using the bijection $\varphi$ and splitting each node $v$ into a source node $v^{\text{out}}$ and a sink node $v^{\text{in}}$. More precisely, define the digraph $D=(V,A)$ as follows.
\begin{eqnarray}
V&:=&\{v^{\text{in}},v^{\text{out}}\mid v\in V_1\cup V_2\},\nonumber\\
A&:=&\{(v^{\text{in}},v^{\text{out}})\mid v\in V_1\cup V_2\}\cup\ \{(s^{\text{out}},\varphi(s)^{\text{in}})\mid s\in S\}\nonumber\\
&&\cup\ \{(u^{\text{out}},v^{\text{in}})\mid (u,v)\in A_1\}\cup\{(u^{\text{out}},v^{\text{in}})\mid (v,u)\in A_2\}. 
\end{eqnarray}
Identifying each element $s\in S$ with the corresponding arc $(s^{\text{out}},\varphi(s)^{\text{in}})$, we have 
\begin{equation}
\begin{array}{l}
I\subset S \text{ is a common base of } M_1 \text{ and } M_2 \text{ if and only if there} 
 \\
 \text{exists }k \text{ arc disjoint paths from }U_1^{\text{in}} \text{ to }U_2^{\text{out}} \text{ in }D \text{ passing through } I. \label{eq:common}
 \end{array}
\end{equation}
Extend $D$ with two extra vertices $r$ (source) and $s$ (sink), and arcs $(r,u^{\text{in}})$ for each $u\in U_1$, arcs $(u^{\text{out}},s)$ for each $u\in U_2$ and finally the arc $(s,r)$.
Let $X$ be the incidence matrix of the resulting digraph $D'=(V',A')$.
Define the flow polytope 
\begin{equation}
Q:=\{f\in \R^{A'}\mid Xf=0, 0\leq f(a)\leq 1, \forall a\in A'\setminus\{(s,r)\}, f((s,r))=k\}.	\label{eq:flow}
\end{equation}

Since $X$ is totally unimodular, $Q$ is integer and $P_1\cap P_2$ is the projection of $Q$ onto the coordinates indexed by $S$. Furthermore, Theorem \ref{tu car} implies that $Q$ has the ICP.
\end{proof}

We end this section with some (open) questions concerning possible extensions of Theorem \ref{gammoids icp}.

Gammoids form a subclass of so-called strongly base orderable matroids. It is known that for any two strongly base orderable matroids, the common base polytope has the integer decomposition property (see \cite{Schrijver}). 
\begin{question}
Does the intersection of two base polytopes of strongly base orderable matroids have the ICP?
\end{question}
In \cite{Sebo} Seb\H o asks whether the \Car \! \! rank of the $r$-arborescence polytope can be bounded by the cardinality of the groundset.
An $r$-arborescence is a common base of a partition matroid and a graphic matroid. 
A partition matroid is a gammoid.
\begin{question}
Does the $r$-arborescence polytope have the ICP?
\end{question}


\end{document}